\documentclass[12pt]{amsart}%{article}%
\usepackage{amsmath,amstext,amsfonts,amsthm,amssymb}
\usepackage{eucal}
\usepackage{diagrams}                                      %
\diagramstyle[scriptlabels,PostScript=dvips]               %
\newarrow{Dash}{}{dash}{}{dash}{}                          %
\newarrow{Equal}=====
\renewcommand{\subsubsection}[1]{\addtocounter{subsubsection}{1}
{\ \\[3pt]\bf \thesubsubsection. \  #1} }
\swapnumbers

\newtheorem{lem}[subsubsection]{Lemma}
\newtheorem{prp}[subsubsection]{Proposition}
\newtheorem{crl}[subsubsection]{Corollary}
\newtheorem{THM}[subsection]{Theorem}

\newtheorem*{Thm}{Theorem}

{  \theoremstyle{definition}
           \newtheorem{dfn}[subsubsection]{Definition}

}

\newcommand{\ad}{\mathrm{ad}}
\newcommand{\Ass}{\mathtt{Ass}}

\newcommand{\Lie}{\mathtt{Lie}}
\newcommand{\LBA}{\mathtt{LBA}}
\newcommand{\tB}{\mathtt{B}}
\newcommand{\tBr}{\mathtt{Br}}

\newcommand{\Coder}{\operatorname{Coder}}
\newcommand{\CM}{\mathtt{CM}}
\newcommand{\Com}{\mathtt{Com}}

\newcommand{\Der}{\operatorname{Der}}
\newcommand{\dg}{\mathtt{dg}}

\newcommand{\Ger}{\mathtt{G}}
\newcommand{\HKR}{\mathtt{HKR}}
\newcommand{\Hom}{\mathrm{Hom}} % straight Hom

\newcommand{\LA}{\mathtt{LA}}
\newcommand{\LM}{\mathtt{LM}}
\newcommand{\LCM}{\mathtt{LCM}}

\newcommand{\rlarrows}{\stackrel{\rTo}{\lTo}}

\newcommand{\wt}{\widetilde}
\newcommand{\Alg}{\mathtt{Alg}}
\newcommand{\Coalg}{\mathtt{Coalg}}

\newcommand{\cO}{\mathcal{O}}
\newcommand{\cP}{\mathcal{P}}
\newcommand{\cQ}{\mathcal{Q}}

\newcommand{\fg}{\mathfrak{g}}
\newcommand{\fm}{\mathfrak{m}}

\newcommand{\Q}{\mathbb{Q}}

\begin{document}
\title[]{Noncommutative unfolding of hypersurface singularity}
\author{Vladimir Hinich}
\address{Department of Mathematics, University of Haifa,
Mount Carmel, Haifa 31905,  Israel}
\email{hinich@math.haifa.ac.il}
\author{Dan Lemberg}
\address{Department of Mathematics, University of Haifa,
Mount Carmel, Haifa 31905,  Israel}
\email{lembergdan@gmail.com}

\begin{abstract}
A version of Kontsevich formality theorem is proven for smooth DG algebras.
As an application of this, it is proven that any quasiclassical datum
of noncommutative unfolding of an isolated surface singularity can be quantized.
\end{abstract}
\maketitle

\section{Introduction}

An isolated hypersurface singularity is a polynomial $f\in k[x_1,\ldots,x_n]$
for which the Milnor number
$$ \mu(f)=\dim k[x_1,\ldots,x_n]/(\frac{\partial f}{\partial x_1},\ldots
\frac{\partial f}{\partial x_n})$$

is finite.

An unfolding of a hypersurface singularity is a family of hypersurface
singularities parametrized by an affine space. From algebraic point of view,
the description of unfoldings of $f$ is nothing but the problem of deformations
of the $k[y]$-algebra $k[x_1,\ldots,x_n]$, with the algebra structure defined
by the assignment $y=f(x_1,\ldots,x_n)$.

In this paper we suggest studying non-commutative unfoldings of
hypersurface singularities, that is deformations of $k[y]$-algebra
$A=k[x_1,\ldots,x_n]$, in the world of associative algebras.

In other words, we are interested in studying the Hochschild cochain
complex of $A$ considered as $k[y]$-algebra, having in mind Kontsevich
Formality theorem as a possible ideal answer.

Were $A$ smooth as $k[y]$-algebra, one could use the version of Formality
theorem proven in~\cite{dtt} which would provide a weak equivalence of
the Hochschild complex to the algebra of polyvector fields. Our case is
only slightly more general: $A$ is quasiisomorphic to a smooth dg algebra
over $k[y]$. Fortunately, the proof of Formality theorem presented
in~\cite{dtt} can be easily generalized to this setup. As a result, we
can replace the Hochschild cochain complex with a certain algebra of
polyvector fields (which in our case is also a dg algebra). 
This considerably simplifies the study of noncommutative unfoldings.

The paper consists of two parts. In the first part (Sections 2,3)
we prove the following version of Formality theorem for smooth dg algebras.

\begin{THM}\label{thm:formality}
Let $R\supset\Q$ be a commutative ring and let $A$ be a commutative smooth
dg $R$-algebra, that is non-positively graded, semifree over $A^0$ which is
smooth as $R$-algebra. Then the Hochschild
cochain complex of $A$ over $R$ is equivalent to the dg algebra of polyvector
fields as (homotopy) Gerstenhaber algebras.
\end{THM}

Recall that for a smooth dg algebra $A$ the algebra of polyvector fields
is defined as $S_A(T[-1])$ where $T=\Der_R(A,A)$ is the $A$-module of 
$R$-derivations of $A$; it is cofibrant when $A$ is as indicated above.

The proof of the theorem is an adaptation
(and simplification) of the proof given in \cite{dtt}.

Since $S_A(T[-1])$ is a Gerstenhaber algebra, its Harrison chain complex 
\newline
$B_{\Com^\perp}(S_A(T[-1]))$ has a structure of dg Lie bialgebra. 
Homotopy Gerstenhaber algebra structure on the Hochschild complex $C(A)$
can be also described via a dg Lie bialgebra structure on 
$F^*_\Lie(C(A)[1])$, see \cite{h-tam}, 6.2 or Subsection~\ref{sss:btilde} below.

An equivalence between Gerstenhaber algerbas of polyvector fields $S_A(T[-1])$ 
and the Hochschild complex $C(A)$ is presented on the level of these 
Lie bialgebra models: we present a dg Lie bialgebra $\xi(A)$  an two weak 
equivalences $\xi(A)\to B_{\Com^\perp}(S_A(T[-1]))$ and
$\xi(A)\to F^*_{\Lie}(C(A))[1])$ of dg Lie bialgebras. The proof of the 
first weak equivalence is straightforward; the second weak equivalence is 
deduced from a dg version of Hochschild-Kostant-Rosenberg theorem; 
however, the setup of dg smooth algebras makes this deduction quite 
nontrivial; this part presented in Subsection~\ref{ss:kappa} is our main 
deviation from the proof of \cite{dtt}.

\subsection{}

In the second part of the paper we apply the formality theorem to studying
noncommutative unfoldings of hypersurface singularities. 

The famous consequence of Kontsevich formality theorem says that any Poisson
bracket on an affine space (or, more generally, on a $C^\infty$ manifold)
can be extended to a star-product.

Poisson bracket appears in this picture as a representative of the first-order
deformation extendable to a second-order deformation.

Having this in mind, we suggest the following
\begin{dfn}
A quasiclassical datum of quantization of a $B$-algebra $A$ is its
deformation over $k[h]/(h^2)$ extendable to $k[h]/(h^3)$.
\end{dfn}
Thus, a quasiclassical datum for a quantization of the ring of smooth functions
on a manifold is precisely a Poisson bracket on the manifold.

Let $f\in k[x_1,\ldots,x_n]$ define an isolated singular hypersurface
and let $W$ be a vector subspace of $k[x_1,\ldots,x_n]$ complement to 
the ideal $(\frac{\partial f}{\partial x_1},\ldots
\frac{\partial f}{\partial x_n})$.

We prove (see Proposition~\ref{prp:QC}) 
that the quasiclassical data for a NC unfolding of an 
isolated singularity $f\in k[x_1,\ldots,x_n]$ are given by pairs $(p,S)$
where $p\in W$ and $S$ is a Poisson vector field satisfying the condition
$[f,S]=0$.

The main result of the second part of the paper
is the following.

\begin{Thm}\label{thm:unfolding}(see Corollary~\ref{crl:unfolding})
Let $f\in k[x,y,z]$ define an isolated surface singularity. Then
any quasiclassical datum of NC unfolding of $f$ can can be quantized to a
noncommutative unfolding over $k[[h]]$.
\end{Thm}

\subsection{Acknowledgement}

We are grateful to Martin Markl
\footnote{Martin did this some 8 years ago}
and Michel Van den Bergh for having pointed out
at an error in the first version of the manuscript. 
We are also very grateful to anonymous referee of PhD thesis of D.L. who found an error 
in the published version of the paper.

\section{Preliminaries}

%Throughout Sections 2 and 3 $k$ is a base commutative ring.

Let $A$ be a commutative $k$-algebra and $T$ a Lie algebroid over $A$. 
Then the symmetric algebra $S_A(T[-1])$ has a natural structure of Gerstenhaber
algebra (in what follows $\Ger$-algebra): the commutative multiplication
is that of the symmetric algebra and the degree $-1$  Lie bracket is induced 
from the Lie bracket on $T$.

The $\Ger$-algebra $S_A(T[-1])$ satisfies an obvious universal property:
given a $\Ger$-algebra $X$, a map $\alpha:A\to X$ of commutative algebras and
a map $\beta:T\to X[1]$ of Lie algebras, so that $\beta$ is also a map of
modules over $\alpha$ and $\alpha$ is a map of $T$-modules via $\beta$,
there is a unique map of $\Ger$-algebras
$S_A(T[-1])\to X$.

Recall that the Hochschild cochain complex $C(A)$ has a $\Ger$-algebra
structure (however, in a  {\sl weak} sense, see \ref{ss:alg-hoch} below).

Thus, we may try using the above universal property to construct
a map $S_A(T[-1])\rTo C(A)$
\footnote{which has a good chance to be quasiisomorphism by 
Hochschild-Kostant-Rosenberg theorem.}
 of $\Ger$-algebras: the pair of obvious maps
\begin{eqnarray*}
A&=&C^0(A)\rTo C(A)\\
T&\rTo& \Hom_k(A,A)=C^1(A)\rTo C(A)[1]
\end{eqnarray*}
should satisfy all necessary properties.

This would give an exceptionally simple proof of Kontsevich formality theorem.

The main obstacle to this plan is that $C(A)$ is not a genuine $\Ger$-algebra;
it has only a structure of Gerstenhaber algebra up to homotopy. This 
obstacle can be, however, overcome, with a bit of Koszul duality
and a standard homotopy theory for colored operads. 

The proof presented below is a result of processing the proof
by Dolgushev, Tamarkin and Tsygan~\cite{dtt}. The main theorem of 
\cite{dtt} is generalized to smooth (non-positively graded) dg algebras
over a commutative $\Q$-algebra.

We have slightly streamlined the agrument working with dg Lie bialgebras 
instead of $\Ger^\perp$-coalgebras. On the other hand, the usage of HKR 
theorem has become more painful in our generalized context
of smooth dg algebras.

\subsection{Colored operads}

The operads appearing in this notes have more than one
\footnote{actually, two} color.
Colored operads were introduced in \cite{BV} back in 70-ies, but are much less 
in use than their colorless version.
 
A colored operad $\cO$ has a set of colors (denoted $[\cO]$) and a collection
of operations $\cO(c,d)$ for any finite collection of colors $c:I\to[\cO]$
and another color $d\in[\cO]$. There is an associative composition
of  operations, and unit elements in $\cO(\{c\},c)$ for all $c\in[\cO]$.

The results of \cite{haha} about model category structure for operads
and operad algebras in complexes extend easily to the colored setup.
In particular,
for $k\supset\Q$ and for any colored operad $\cO$ in complexes over $k$,
the category $\Alg_\cO$ of $\cO$-algebras has a model structure with 
quasiisomorphisms as
weak equivalences and componentwise surjections as fibrations. The category
of colored operads with a fixed collection of colors is itself the category
of algebras over a certain colored operad, therefore a model structure
on operads in characteristic zero.

A map $f:\cP\to\cQ$ of operads induces a forgetful functor
$f^*:\Alg_\cQ\rTo\Alg_\cP$ and its left adjoint 
$f_!:\Alg_\cP\rTo\Alg_\cQ$. This is a Quillen pair; it is a Quillen equivalence
if $f:\cP\to\cQ$ is a quasiisomorphism.

The above claims are proven for colorless operads in \cite{haha}.
Their colored versions can be found in \cite{strict} (paper in preparation).

\subsection{Koszul duality}

The material of this subsection is standard. The details can be found in
\cite{gk}, \cite{gj}, \cite{h-tam}, \cite{vdl}, \cite{LV}.

\subsubsection{}
Let us recall some standard notation connected to Koszul duality of operads.
Let $k$ be a field of characteristic zero. Let $\cO$ be a (possibly colored) 
Koszul operad in graded vector spaces over $k$ and $\cO^\perp$ the 
corresponding quadratic dual cooperad.

Any $\cO$-algebra $X$ gives rise to a differential in the cofree $\cO^\perp$
coalgebra $F^*_{\cO^\perp}(X)$. Dually, any $\cO^\perp$-coalgebra $Y$
defines a differential on the free $\cO$-algebra $F_\cO(Y)$. These assignments
define a pair of adjoint functors
\begin{equation}\label{eq:KD}
\Omega_{\cO}:\Coalg_{\cO^\perp}\rlarrows \Alg_{\cO}:B_{\cO^\perp}.
\end{equation}
A map of $\cO$-algebras is called weak equivalence if it is a quasiisomorphism.
A map of $\cO^\perp$-coalgebras is called  weak equivalence if the functor
$\Omega_{\cO}$ carries it to a quasiisomorphism.

The unit and the counit of the adjunction are weak equivalences; in particular,
the map $\Omega_\cO\circ B_{\cO^\perp}(A)\rTo A$ is a quasiisomorpism.

\subsubsection{$\cO_\infty$-algebras}

Here $\cO$ is a colored Koszul operad. By definition, an $\cO_\infty$-algebra
structure on a graded vector space $X$ is just a differential on the graded
$\cO^\perp$-coalgebra $F^*_{\cO^\perp}(X)$ converting it into a
dg $\cO^\perp$-coalgebra. This differential is defined by a collection of maps
$$ d_n:F^{*n}_{\cO^\perp}(X)\to X[1]$$
satisfying the condition expressing the property $d^2=0$.
The component $d_1$ yields a differential on $X$.

We use the notation $B_{\cO^\perp}(X)$ for the
differential graded $\cO^\perp$-coalgebra $(F^*_{\cO^\perp}(X),d)$.

Any dg $\cO$-algebra $X$ can be considered as a
$\cO_\infty$-algebra, so a canonical map

\begin{equation}
\cO_\infty\rTo\cO
\end{equation}
of dg operads is defined. It is a quasiisomorphism.

There are two different notions of morphism of $\cO_\infty$-algebras.
The first is just a morphism of algebras over the dg operad $\cO_\infty$.
This is a map of complexes $f:X\to Y$ preserving the $\cO_\infty$-algebra
structure. The second, more general, is called an $\cO_\infty$-morphism and
it is defined as a morphism $F:B_{\cO^\perp}(X)\to B_{\cO^\perp}(Y)$ of dg
coalgebras.

It is defined by its components $F_n:F^{*n}_{\cO^\perp}(X)\rTo Y$ satisfying some
quadratic identities. An $\cO_\infty$-morphism $F$ is called a weak equivalence
if it is a weak equivalence of the dg $\cO^\perp$-coalgebras. One can easily
check that $F$ is a weak equivalence iff $F_1:X\to Y$ is a quasiisomorphism.

If $A$ is an $\cO_\infty$-algebra, $B_{\cO^\perp}(A)$ is a dg $\cO^\perp$-coalgebra
and one has an $\cO_\infty$-weak equivalence
$$ A\rTo \Omega_\cO\circ B_{\cO^\perp}(A)$$
whose first component is a quasiisomorphism of complexes described,
if one forgets the differential, as the composition of
$A=F^{*1}_{\cO^\perp}(A)\to F^*_{\cO^\perp}(A)$ with
$B_{\cO^\perp}(A)=F^1_{\cO}(B_{\cO^\perp}(A))\to F_{\cO}(B_{\cO^\perp}(A)).$

\subsubsection{Examples}

The following operads are Koszul.
\begin{itemize}
\item $\cO=\Com,\Ass,\Lie$ with $\cO^\perp=\Lie^*\{1\},\Ass^*\{1\},\Com^*\{1\}$.
\item $\Ger$, the operad for Gerstenhaber algebras, with
$\Ger^\perp=\Ger^*\{2\}$.
\item $\cO=\CM$, the two-color operad governing pairs $(A,M)$ where $A$ is a
commutative algebra and $M$ is an $A$-module. Similarly, $\LM$ is the
two-color operad governing pairs $(L,M)$ where $L$ is a Lie algebra and $M$ is
an $L$-module. Both operads are Koszul with $\CM^\perp=\LM^*\{1\}$
and $\LM^\perp=\CM^*\{1\}$.
\item $\LA$, the two-color for Lie algebroids. An $\LA$-algebra is a pair
$(A,T)$ consisting of a commutative algebra $A$ and a Lie algebroid $T$.
The operations include, apart of the commutative multiplication on $A$ and
a Lie bracket on $T$, an $A$-module structure on $T$ and a $T$-module
structure on $A$. One has $\LA^\perp=\LA^*\{1\}$. $\LA$  is a Koszul operad
(see \cite{vdl}), but we will not use this fact.
\end{itemize}

\subsubsection{A slight generalization}

Let  $R$ be a commutative $k$-algebra. Given an operad $\cO$ over $k$,
it makes sense to talk about $\cO$-algebras with values in the category of 
complexes $\dg(R)$.
The category of such algebras is denoted $\Alg_{\cO}(\dg(R))$. 
If $\cO$ is Koszul,
one still has an adjoint pair
\begin{equation}
\Omega_\cO:\Coalg_{\cO^\perp}(\dg(R))\rlarrows\Alg_\cO(\dg(R)):B_{\cO^\perp},
\end{equation}
defined by the same formulas as for $R=k$ but using the symmetric monoidal
category $\dg(R)$ of complexes over $R$ instead of that over $k$.
The canonical map
$$ \Omega_\cO\circ B_{\cO^\perp}(A)\to A$$
is still a weak equivalence for each $A\in\Alg_{\cO}(\dg(R))$.

The notions of $\cO_\infty$-algebra and of $\cO_\infty$-morphism extend
without difficulty to algebras in $\dg(R)$.

\subsection{Algebra structure on Hochschild cochain complex}
\label{ss:alg-hoch}
\subsubsection{$\wt\tB$-algebras}
\label{sss:btilde}

Let $X$ be a $\Ger$-algebra. Then $X[1]$ has a Lie algebra structure, so that
$B_{\Com^\perp}(X)[1]$ which is $F^*_\Lie(X[1])$ considered as a graded vector
space,  acquires a dg Lie bialgebra structure.

Vice versa, any dg Lie bialgebra structure on $F^*_\Lie(X[1])$ gives rise to
a $\Ger_\infty$-structure on $X$.

This leads to definition of another dg operad $\wt\tB$ whose action on a
complex $X$ is given by a dg Lie bialgebra structure $F^*_\Lie(X[1])$
extending the standard Lie coalgebra structure and the differential on $X$.

Since any $\Ger$-algebra has a natural structure of $\wt\tB$-algebra, and any
$\wt\tB$-algebra structure on $X$ extends to a  $\Ger_\infty$-algebra, one has
a decomposition
\begin{equation}
\Ger_\infty\rTo\wt\tB\rTo\Ger,
\end{equation}
of the canonical map $\Ger_\infty\to\Ger$.

The $\wt\tB$-structure on $X$ is given by the collection of the following
operations:
\begin{itemize}\label{eq:elld-1}
\item $\ell_{m,n}:F^{*m}_\Lie(X[1])\otimes F^{*n}_\Lie(X[1])\to X[1]$,
\item $d_n:F^{*n}_\Lie(X[1])\to X[2]$,
\end{itemize}
defining the Lie bracket and the differential on $F^*_\Lie(X[1])$, subject
to certain relations which assure that $d^2=0$, $d$ is a derivation
of the bracket, and the cocycle condition connecting the bracket with the
cobracket.

Note for book-keeping the degrees of $\ell_{m,n}$ and $d_n$.
\begin{itemize}
\item $\ell_{m,n}:\Lie(m)^*\otimes\Lie(n)^*\rTo\wt\tB(m+n)^{1-m-n},\ m,n\geq 1$,
\item $d_n:\Lie(n)^*\rTo\wt\tB(n)^{2-n},\ n\geq 2$.
\end{itemize}

\subsubsection{Lie bialgebras versus $\Ger$-algebras}

The operad $\Ger$ is Koszul, so we have a standard Koszul duality pair of
adjoint functors

\begin{equation*}
\Omega_{\Ger}:\Coalg_{\Ger^\perp}\rlarrows \Alg_{\Ger}:B_{\Ger^\perp}.
\end{equation*}

There is another pair of adjoint functors, a sort of
``relative Koszul duality'', based on the fact expressed in \ref{sss:btilde}:
if $X\in\Alg_\Ger$, the dg Lie coalgebra $B_{\Com^\perp}(X)[1]$
has a structure of Lie bialgebra. Dually, given a dg Lie bialgebra $Y$,
the commutative algebra $\Omega_{\Com}(Y[-1])$ has a structure of $\Ger$-algebra.

This leads to the pair of adjoint functors
\begin{equation}
\Omega'_{\Com}:\LBA\rlarrows \Alg_{\Ger}:B'_{\Com^\perp},
\end{equation}
where $\LBA$ denotes the category of dg Lie bialgebras and
$$ B'_{\Com^\perp}(X)=B_{\Com^\perp}(X)[1]\textrm{ and }
\Omega'_{\Com}(Y)=\Omega_{\Com}(Y[-1]).$$

As for the conventional Koszul duality (\ref{eq:KD}) an arrow in $\LBA$
will be called a weak equivalence iff its image under $\Omega'_{\Com}$ is a
quasiisomorphism.

We use the same notation $B'_{\Com^\perp}$ for the obvious extension of the
functor to $\tilde\tB$-algebras.

\subsubsection{Deligne conjecture}
\label{ss:deligne-conjecture}

Deligne conjecture asserts that the cohomological Hochschild complex $C(A)$
of an associative algebra $A$ has a structure of an algebra over an operad
of (chains of) little squares. Even though Deligne conjecture is very
much relevant for the Formality theorem, the version we need is
extremely easy.

Define a $\tB_\infty$-algebra structure on a graded vector space $X$
as the structure of dg bialgebra on the free associative coalgebra
$F^*_{\Ass}(X[1])$. Similarly to $\tilde\tB$-algebras, this leads to a dg operad
$\tB_\infty$ governing such algebras. This operad is generated by the operations

\begin{itemize}\label{eq:b-infty}
\item $m_{p,q}:X^{\otimes p}\otimes X^{\otimes q}\rTo X[1-p-q]$, the components
of the product, and
\item $m_n:X^{\otimes n}\rTo X[2-n],$ the components of the differential,
\end{itemize}
defining the associative multiplication and the differential on 
$F^*_\Ass(X[1])$, subject
to certain relations which assure that $d^2=0$, $d$ is a derivation
of the bracket, and the condition describing compatibility of the product
with the coproduct.

The Hochschild complex $C(A)$ has a canonical $\tB_\infty$-algebra structure 
defined by the formulas:
\begin{itemize}
\item $m_2$ is the cup-product.
\item $m_k=0$ for $k>2$.
\item $m_{1,l}$ are the brace operations
 $x_0,\ldots,x_l\mapsto x_0\{x_1,\ldots,x_l\}$ having degree $-l$.
\item $m_{k,l}=0$ for $k>1$.
\end{itemize}

The associative cup-product together with the brace operations generate
an operad $\tBr$ called the operad of braces. Thus, the action of 
$\tB_\infty$ on $C(A)$ factors through $\tBr$ whose action on $C(A)$ is 
more or less tautological.

It turns out that the operad $\tBr$ is equivalent to the operad of small
squares, so the action of $\tBr$ on $C(A)$ described above ``solves''
Deligne conjecture.

What is much more important for us is that there exists a canonical
map of  operads $\wt\tB\rTo\tB_\infty$ (depending of a choice of associator)
so that any $\tB_\infty$-algebra is endowed with a canonical $\wt\tB$-algebra 
structure.
This remarkable result was proven by 
Tamarkin in his 1998 paper on Kontsevich formality theorem, 
see~\cite{tk}, \cite{tt}, \cite{h-tam}.
The proof is based on Etingof-Kazhdan theory of quantization 
(and dequantization) of Lie bialgebras.
\footnote{We have no doubt that the maps $\wt\tB\to\tB_\infty\to\tBr$
are quasiisomorphisms; unfortunately we were unable to find a reference for 
this fact.
}

\section{Equivalence of Lie bialgebra models}

In this section we are working in the symmetric monoidal category of
complexes over a commutative ring $R\supset\Q$.

$A$ is a smooth dg algebra over $R$ and $C=C(A)$ is the Hochschild cochain 
complex of $R$-algebra $A$. All operads considered will live over $\Q$; 
all our $\cO$-algebra will be in $\Alg_\cO(\dg(R))$.

\subsection{}
\label{ss:eq-intro}
According to the above, the Hochschild complex $C=C(A)$ admits a
$\tilde\tB$-algebra structure expressible (in a very nontrivial way)
via the cup product and the brace operations on $C$.

The corresponding  dg Lie bialgebra structure on $F^*_{\Lie}(C[1])$ is given by
the collection of maps $\ell_{m,n},\ d_n$ described in (\ref{eq:elld-1}),
and together they form a dg Lie bialgebra denoted $B'_{\Com^\perp}(C)$.

The algebra of polyvector fields $S_A(T[-1])$ is a (strict) $\Ger$-algebra,
so it leads to dg Lie bialgebra $B'_{\Com^\perp}(S_A(T[-1]))$. In order to prove
the main theorem, we will present a pair of weak equivalences

\begin{equation}
B'_{\Com^\perp}(S_A(T[-1])) \lTo^\iota \xi(A) \rTo^\varkappa B'_{\Com^\perp}(C(A))
\end{equation}
in the category of dg Lie bialgebras.

We will proceed as follows.

First of all we identify a dg Lie coalgebra $\xi(A)$ which naturally maps to
$B'_{\Com^\perp}(S_A(T[-1]))$. We can easily check the map is a weak equivalence,
it is injective, and that its image is closed with respect to Lie bracket.
This endows $\xi(A)$ with a structure of Lie bialgebra.

On the other hand, we will see that the pair of obvious embeddings 
$\alpha:A\to C$ and $\beta:T\to C[1]$ induces a map of dg Lie bialgebras
$\xi(A)\to B'_{\Com^\perp}(C(A))$. Finally, the fact that it is a weak equivalence
follows from Hochschild-Kostant-Rosenberg theorem.

\subsection{}

There is a pair of adjoint functors
\begin{equation}
F:\Alg_{\CM}\rlarrows\Alg_{\Com}:G
\end{equation}
defined by the formulas
$$ G(A)=(A,A[1]);\quad F(A,M)=S_A(M[-1]).
\footnote{
The symmetric algebra $S_A(M)=\bigoplus_{n\geq 0} S^n_A(M)$ makes sense even
if the commutative algebra $A$ has no unit. This is important as our operads
are non-unital.} 
$$
On the Koszul-dual side, there is a pair of adjoint functors
\begin{equation}
f:\Coalg_{\CM^\perp}\rlarrows\Coalg_{\Com^\perp}:g
\end{equation}
defined by the formulas
$$g(C)=(C,C[1]);\quad f(C,N)=C\oplus N[-1],$$
with the cobracket on $f(C,N)$ determined by the cobracket on $C$
and the coaction $\delta:N\to N\otimes C$ so that the value of the cobracket
at $x\in M[-1]\subset f(C,N)$ is $\delta(x)-\sigma\circ\delta(x)$
where $\sigma:C\otimes N[-1]\to N[-1]\otimes C$ is the standard commutativity
constraint.

The functors $G$ and $g$ commute with the Bar-construction, so that the
compositions $B_{\CM^\perp}\circ G$ and $g\circ B_{\Com^\perp}$ are naturally
isomorphic. This yields the composition
\begin{equation}\label{eq:baby}
f\circ B_{\CM^\perp}\rTo f\circ B_{\CM^\perp}\circ G\circ F=
f\circ g\circ B_{\Com^\perp}\circ F\rTo B_{\Com^\perp}\circ F.
\end{equation}

To get a feeling of what is going on, let $(A,M)\in\Alg_{\CM}$. The
composition $B_{\Com^\perp}\circ F$ applied to $(A,M)$ gives the (shifted)
dg Lie coalgebra Koszul dual to the commutative algebra $S_A(M[-1])$
which is graded by powers of $M$.  The composition
$f\circ B_{\CM^\perp}$ applied to $(A,M)$ is the dg subcoalgebra
of $B_{\CM^\perp}(S_A(M[-1]))$ consisting of the elements of degree $\leq 1$.

In particular, the map~(\ref{eq:baby}) is injective.

\subsection{}
We wish to apply the map of functors~(\ref{eq:baby}) to a Lie algebroid
$(A,T)$. The functor $F$ applied to a Lie algebroid, yields a $\Ger$-algebra,
so we upgrade it to the functor
$$ F':\Alg_{\LA}\to\Alg_{\Ger}.$$

In the diagram below we draw the categories and the functors described above.
The arrows denoted $\#$ are forgetful functors ($\#[-1]$ is the composition
of the forgetful functor with a shift).

The diagram looks more symmetric if one adds an extra vertex which we denote
$\LCM$. This is the category of $\CM^\perp$-coalgebras $X$ together
with a Lie bialgebra structure on $f(X)[1]$. One has a forgetful functor
$\LCM\to\Coalg_{\CM^\perp}$ and an obvious functor $f':\LCM\to\LBA$.
The functor $B'_{\CM^\perp}$ is defined later on, see Lemma~\ref{lem:xiisLBA}
and a discussion after it.

\begin{equation}\label{para}
\begin{diagram}
\Alg_{\LA}&& \rTo^{F'} && \Alg_{\Ger}&& \\
& \rdTo^{\#} &&& \uTo\vLine_{B'_{\Com^\perp}} & \rdTo^{\#} &\\
\dDashto^{B'_{\CM^\perp}}   &&\Alg_{\CM}&\pile{\rTo^F\\ \lTo_G} &\HonV&&\Alg_{\Com}\\
&&\uTo^{\Omega_\CM} \dTo_{B_{\CM^\perp}} &&\vLine^{\Omega'_\Com}\dTo&&\\
\LCM& \rDash & \VonH & \rDashto^{f'} &\LBA&&\uTo^{\Omega_\Com}\dTo_{B_{\Com^\perp}} \\
&     \rdDashto_{\#}      &&&& \rdTo_{\#[-1]} &\\
&&\Coalg_{\CM^\perp}&&\pile{\rTo^f\\\lTo_g}&&\Coalg_{\Com^\perp}\\
\end{diagram}
\end{equation}

Recall that the fuctors $G$ and $g$ (see the front face of the cube)
commute with the Bar-constructions, which leads to a canonical morphism
of functors
$$
f\circ B_{\CM^\perp}\rTo B_{\Com^\perp}\circ F.
$$

\

Let now $(A,T)$ be a Lie algebroid.
We define $\xi(A,T)=f(B_{\CM^\perp}(A,T))[1]$. By definition, this is a dg Lie
coalgebra and one has a canonical injective map
\begin{equation}
\iota: \xi(A,T)\rTo B_{\Com^\perp}(F(A,T))[1]=B'_{\Com^\perp}(F'(A,T)).
\end{equation}
Recall that $F'(A,T)=S_A(T[-1])$ is precisely what we need, so we have
constructed a dg Lie subcoalgebra $\xi(A,T)$. It is easy to check
(see Lemma~\ref{lem:xiisLBA} below) that $\xi(A,T)$ is also a Lie subalgebra,
so that $\xi(A,T)\in\LBA$. But even before doing this, let us check that
the embedding $\iota:\xi(A,T)\to B_{\Com^\perp}(F(A,T))[1]$ is a weak equivalence
of dg Lie coalgebras.

In fact, one has a weak equivalence
$\Omega_\CM\circ B_{\CM^\perp}(A,T)\to (A,T)$. Applying the functor $F$ and we
get a weak equivalence
\begin{equation}
F\circ\Omega_\CM\circ B_{\CM^\perp}(A,T)\to F(A,T).
\end{equation}
Since the map $\Omega_\Com\circ B_{\Com^\perp}(F(A,T))\to F(A,T)$ is also
a weak equivalence, the commutative diagram
\begin{equation}
\begin{diagram}
\Omega_\Com\circ f \circ B_{\CM^\perp}(A,T) & \rTo^{\Omega_\Com(\iota)} &
\Omega_\Com\circ B_{\Com^\perp}\circ F(A,T) \\
\dEqual & & \dTo \\
F\circ\Omega_\CM\circ B_{\CM^\perp}(A,T) &\rTo & F(A,T)
\end{diagram}
\end{equation}
asserts that $\iota$ is a weak equivalence in $\Coalg_{\Com^\perp}$.

\begin{lem}\label{lem:xiisLBA}
The image of $f(B_{\CM^\perp}(A,T))$ in $B'_{\Com^\perp}(F'(A,T))$
is a Lie subalgebra.
\end{lem}

\begin{proof}[Proof of the lemma]
Denote $V=F'(A,T)=S_A(T[-1])=\oplus S_A^n(T[-1])$.
The Lie bialgebra $B'_{\Com^\perp}(V)$ as a graded space is just
$$ F^*_{\Lie}(V[1])=\bigoplus_{n\geq 1}(\Lie(n)^*\otimes V^{\otimes n}[n])^{S_n}.
$$
The Lie bracket on it is extended from the Lie bracket on $V[1]$. The space
$V$ is graded, and this grading induces a grading on  $F^*_{\Lie}(V[1])$.
The image of $f(B_{\CM^\perp}(A,T))$ consists of elements having degree $\leq 1$.
The Lie gracket has degree $-1$ with respect to this grading; therefore,
the image is closed with respect to the Lie bracket.
\end{proof}

\subsection{}\label{ss:2.3}

From now on $A$ is a smooth dg commutative algebra over $R\supset\Q$
and $T=\Der_R(A,A)$.
This means that $A^0$ is a smooth commutative $R$-algebra and the map
$A^0\to A$ is a finitely generated cofibration 
(that is, $A$ is generated as a graded $A^0$-algebra
by a finite number of free variables $x_i$ of negative degree).

In what follows we will write $\xi(A)$ instead of $\xi(A,T)$.
According to \ref{ss:eq-intro} the
complex $B'_{\Com^\perp}(C)=(F^*_{\Lie}(C[1]),d)$
has a structure of dg Lie bialgebra. We will present a Lie bialgebra map
$\varkappa:\xi(A)\to B'_{\Com^\perp}(C)$ and prove it is a weak equivalence.

Our plan is as follows. First of all we will present a map of dg Lie coalgebras
$\varkappa:\xi(A)\to B'_{\Com^\perp}(C)$, then we will check it is a Lie algebra
homomorphism, and after that we will check it is a weak equivalence of Lie
bialgebras.

To present a map of dg Lie coalgebras, it suffices to have a map
$$B_{\CM^\perp}(A,T)\rTo g(B'_{\Com^\perp}(C)).$$
The latter is given by a pair of maps $(\alpha,\beta)$ where
\begin{equation}
\alpha:A\rTo C
\end{equation}
is a map of $\Com_\infty$ algebras and
\begin{equation}
\beta:T\rTo C[1]
\end{equation}
is a map of $\Com_\infty$-modules over $\alpha$.
The maps are precisely the maps we were talking about from the very beginning,
$\alpha:A\to C^0(A)=A$ and $\beta:T\to\Hom(A,A)=C(A)[1]^0$.

To check that $\alpha$ induces a map of $\Com_\infty$-algebras,
we need to check that the map $\alpha:A\to C$ induces a map of the
Bar-constructions which commutes with the differentials.
This is equivalent to checking that the higher
components of the differential
$$ d_n:F^{*n}_{\Com^\perp}(C)\to C[1],\ n> 3,$$
vanish on $F^{*n}_{\Com^\perp}(A)$ and $d_2$ coincides with the multiplication
in $A$.

Similarly, in order to check that $\beta$ is a map of $\Com_\infty$-modules
over $\alpha$, one needs to verify that $d_n,\ n>2$ also vanish on
the part of $F^{*n}_{\Com^\perp}(C)$ having $n-1$ component $C^0$ and one
component $C^1$.

Both statements are independent of $A$; they are verified in Theorem 3 of
\cite{dtt}.

Thus, we already know that
$\varkappa:\xi(A)\to B'_{\Com^\perp}(C)$ is a map of dg
Lie coalgebras. To check it preserves the bracket, it suffices to
compose $\varkappa$ with the projection to cogenerators $C[1]$ of
$B'_{\Com^\perp}(C)$.
\footnote{
Let us explain the last point. The bracket map $X\otimes X\to X$ in a Lie
bialgebra is a coderivation, for an appropriate notion of coderivation
from a comodule to a Lie coalgebra. A composition of a coderivation
with a homomorphism of Lie coalgebras gives a coderivation.
We have to compare two coderivations into a cofree Lie coalgebra. It is
sufficient to compare their corestrictions on the cogenerators.
}

Since the projection itself is a Lie algebra homomorphism, we have to
verify that the composition
\begin{equation}\label{eq:comp}
\xi(A)\rTo^\varkappa B'_{\Com^\perp}(C)\rTo C[1]
\end{equation}
is a Lie algebra homomorphism.

We can forget about the differentials.
The inclusion $\xi(A)\to F^*_{\Com^\perp}(S_A(T[-1]))$ induces a grading on
$\xi(A)$; the composition $\xi(A)\to C[1]$ is zero on all components
except for degree $1$. Thus, it factors as
$$\xi(A)\to A\oplus T[-1] \to C[1],$$
so it remains to check that the obvious map
\begin{equation}\label{eq:ab}
A\oplus T[-1]\to C[1]
\end{equation}
preserves the Lie bracket.
This is obvious when $C[1]$ is endowed with the Gerstenhaber bracket.
The rest follows from
\begin{prp}
The $\Lie_\infty$-structure on C[1] defined by the $\wt\tB$-structure,
coincides with the (strict) Gerstenhaber bracket.
\end{prp}
\begin{proof}
The claim is independent of $A$ and is precisely Theorem 2 of \cite{dtt}.
\end{proof}

The map $\varkappa:\xi(A)\to B'_{\Com^\perp}(C)$ is therefore a map of
Lie bialgebras.

\subsection{$\varkappa$ is a weak equivalence}\label{ss:kappa}

We have to verify that $\varkappa:\xi(A)\to B'_{\Com^\perp}(C)$
is a weak equivalence, that is that the map
\begin{equation}\label{eq:Omegakappa}
\Omega_\Com(\varkappa):
F\circ\Omega_{\CM}\circ B_{\CM^\perp}(A,T)=
\Omega_{\Com}\circ f\circ B_{\CM^\perp}(A,T)\to
\Omega_{\Com}\circ B_{\Com^\perp}(C)
\end{equation}
is a quasiisomorphism. We will deduce this from a dg version of
Hochschild-Kostant-Rosenberg (HKR) theorem for smooth dg algebras.

Let $A$ be a smooth commutative dg algebra over $R\supset\Q$. According to~\cite{l}, 5.4.5.1, the homological
HKR map $C_\bullet(A,A)\rTo S_A(\Omega_A[1])$ is a quasiisomorphism. This is a map
of cofibrant $A$-modules, so it induces a quasiisomorphism of complexes
\begin{equation}\label{eq:HKR}
\HKR:S_A(T[-1])\rTo C=\Hom(C_\bullet(A,A),A).
\end{equation}
Note that the cohomology of $\HKR$ is compatible with the Gerstenhaber 
structures.

This immediately implies $\Omega_\Com(\varkappa)$ is a quasiisomorpism
in the case $A$ has trivial differential, for instance,
when $A$ is a (conventional) smooth algebra.
In fact, the map (\ref{eq:Omegakappa}) induces in cohomology a
commutative algebra homomorphism from $S_A(T[-1])$ to $H(C)$ which coincides
with $\HKR$ on $A$ and on $T$. Then by HKR theorem it is a quasiisomorphism.

In general, $H(C)$ needs not be generated by $H(T)$ over $H(A)$, so the 
above reasoning does not work. But, as it turns out, it can be easily fixed
using the notion of $H$-commutative algebra explained below.

\subsubsection{$H$-commutative algebras}

An $H$-commutative algebra over $R$ is just a commutative algebra in the derived 
category $D(R)$. \footnote{The term should remind $H$-spaces in topology.}

The homotopy theory of operads and of operad algebras teaches us that the notion of
algebra in the derived category is meaningless and should be replaced with a notion
of algebra over an operad equivalent to $\Com$. But what we are doing here is just the 
opposite.

A map of $H$-commutative algebras is just a map in $D(R)$ preserving the multiplication.

Any $\Com_\infty$-algebra $A$ has a canonical $H$-commutative algebra structure
given by the degree zero binary operation. Any $\Com_\infty$ map $f:A\to B$ of 
$\Com_\infty$ algebras, gives rise to a map of $H$-commutative algebras --- it is given 
by the degree zero operation $f_1:A\to B$ which is a map of complexes and commutes with 
the multiplication in $A$ and $B$ up to homotopy.

The Hochschild cochain complex $C$ has an $H$-commutative structure determined 
by the cup-product which is commutative up to homotopy. The same $H$-commutative algebra
structure is defined on $C$ by the choice of morphism $\wt B\to \tBr$:
the symmetrization of cup-product is the only degree zero commutative binary operation 
in $\tBr$, up to a constant, which has to be equal to $1$ as the morphism
$\wt B\to\tBr$ induces the standard Gerstenhaber algebra structure on the cohomology.

The Hochschild-Kostant-Rosenberg map (\ref{eq:HKR}) induces an isomorphism of $H$-
commutative algebras.

We claim that the map $\Omega_\Com(\varkappa):\Omega_\Com(\xi)\to\Omega_\Com\circ 
B_{\Com^\perp}(C)$ induces an isomorphic map of $H$-commutative algebras.

This will imply that $\varkappa$ is a weak equivalence.

\

Look at the diagram 
\begin{equation}
\begin{diagram}
F\circ\Omega_{\CM}\circ B_{\CM^\perp}(A,T)&=&
\Omega_{\Com}\circ f\circ B_{\CM^\perp}(A,T)&\rTo^{\Omega_\Com(\varkappa)}&
\Omega_{\Com}\circ B_{\Com^\perp}(C)\\
\dTo_\sim &&&& \uTo_\sim\\
F(A,T) && \rTo^{\HKR}&& C
\end{diagram},
\end{equation}
where the left vertical equivalence is induced by the counit of the adjunction
$\Omega_\CM\circ B_{\CM^\perp}(A,T)\to (A,T)$, and the right vertical map
is the $\Com_\infty$-equivalence induced by the unit of adjunction.

The diagram defines two maps of $H$-commutative algebras from $F(A,T)=S_A(T[-1])\to C$. The restrictions of these two maps to $A$ and $T$ coincide.

Since $S_A(T[-1])$ is generated by $T$ over $A$ as an $H$-commutative algebra, two maps coincide in $D(R)$.

\section{Application: non-commutative unfolding}

\subsection{}
Let $k$ be a field of characteristic zero and let $f$ be a polynomial
in $A=k[x_1\ldots,x_n]$. We put $B=k[y]$ and we define a $B$-algebra
structure on $A$ via $y=f(x_1,\ldots,x_n)$.

Denote $P=B[x_1,\ldots,x_n,e]$ the semifree $B$-algebra generated by $x_i$ in
degree $0$, $e$ in degree $-1$, with the differential defined as
$$de=f(x_1,\ldots,x_n)-y.$$
The obvious projection $\pi:P\to A$ carrying $x_i$ to $x_i$ and $e$ to $0$,
is a quasiisomorphism. It is split as a homomorphism of $k$-algebras
by $\iota:A\to P$ defined by $\iota(x_i)=x_i$.

\begin{lem}$A$ is free as $B$-module.
\end{lem}
\begin{proof}
It is a standard fact that there is an automorphism of $A$ given by the formulas
\begin{equation}
x_i\mapsto x_i+x_n^{N_i},\quad x_n\mapsto x_n,
\end{equation}
for suitable $N_i$, such that the image of $f$ is a monic polynomial in $x_n$
with coefficients in $k[x_1,\ldots,x_{n-1}]$. This allows one to assume, 
without loss of generality, that $f$ is monic in $x_n$. 
In this case the sequence of elements in $A$
$$ x_1,\ldots,x_{n-1},f$$
is regular and $A$ is free over $k[x_1,\ldots,x_{n-1},f]$. This implies
that $A$ is also free as $B=k[f]$-module.
\end{proof}

\subsection{Comparison of Hochschild complexes}

Let $k$ be a commutative ring, $A$ and $A'$ two dg $k$-algebras cofibrant
as complexes over $k$. We are going to show that if $A$ and $A'$ are 
quasiisomorphic, then their Hochschild cochain complexes are 
quasiisomorphic as dg Lie algebras.

Note that according to a deep result of Keller \cite{ke} the Hochschild cochain
complexes $C(A)$ and $C(A')$ should be equivalent as $B_\infty$-algebras.
We present below a much more elementary result so as not to be compelled
to extend \cite{ke} to dg setup.

Let $A$ be a dg algebra over $k$. The Hochschild cochain complex $C(A)$
can be defined as follows.

We endow a unital cofree dg associative coalgebra
$A^\vee=\oplus_{n\geq 0} A^{\otimes n}[n]$ with a differential encoding
the differential and multiplication in $A$.
The (graded) coderivations of $A^\vee$ form a dg Lie algebra which is
precisely $C(A)[1]$.

Let $f:A\to A'$ be a surjective quasiisomorphism of dg algebras over $k$.
Assume furthermore that both $A$ and $A'$ are cofibrant as complexes over $k$.

Let us show that the Hochschild complexes $C(A)$ and $C(A')$ are equivalent
as dg Lie algebras.

The map $f$ induces a map of dg coalgebras
$$ f^\vee:A^\vee\rTo A^{\prime\vee}.$$
This yields the pair of maps $\phi$ and $\psi$ in the diagram
\begin{equation}
\begin{diagram}
X & \rDashto^{\psi'} & \Coder(A^\vee) \\
\dDashto^{\phi'} & & \dTo^\phi \\
\Coder(A^{\prime\vee}) & \rTo^\psi & \Coder^f(A^\vee,A^{\prime\vee}),
\end{diagram}
\end{equation}
where $\Coder^f$ is the collection of maps $\delta:A^\vee\to A^{\prime\vee}$
satisfying the condition
$$ \Delta\circ\delta=(\delta\otimes f+f\otimes\delta)\circ\Delta.$$
Define $X$ by the cartesian diagram above. Then $X$ inherits the dg Lie algebra
structure. The maps $\phi$ and $\psi$ are both quasiisomorphisms and $\phi$
is surjective, so the maps $\phi'$ and $\psi'$ are quasiisomorphism of dg
Lie algebras.

\begin{crl}
Let $A$ and $A'$ be two dg algebras over $k$ which are cofibrant as complexes.
If $A$ and $A'$ are quasiisomorphic, their Hochschild complexes $C(A)[1]$
and $C(A')[1]$ are quasiisomorphic as dg Lie algebras.
\end{crl}
\begin{proof}
Any pair of quasiisomorphic algebras can be connected by a pair of
surjective quasiisomorphisms from a cofibrant algebra which  is automatically
cofibrant as a complex of $k$-modules.
\end{proof}

\subsection{ }

We are now back to our unfoldings.
According to the above, the dg Lie algebra governing deformations of
$B$-algebra $A$, is the algebra of polyvector fields $S_P(T_P[-1])[1]$
where $T_P=\Der_B(P)$. In a more detail, $T_P$ is is a $P$-module freely
generated by the elements $\partial_i=\frac{\partial}{\partial x_i}$ of
degee $0$
and the element $\partial_e=\frac{\partial}{\partial e}$ of degree $1$,
with the differentials given by the formula
\begin{equation}\label{eq:diff}
d(\partial_e)=0;\quad
d(\partial_i)=\frac{\partial f}{\partial x_i}\partial_e.
\end{equation}
It is convenient to compare $T_P$ with a dg Lie algebroid $T$ over $A$
generated by the same
 $\partial_i$ and $\partial_e$
over $A$, with the differential given by (\ref{eq:diff}).

Note the following
\begin{lem}\label{lem:inner}
The differential in $T$ is inner, given by the formula
$$ d(x)=-[f\partial_e,x].$$
\end{lem}
\qed

The Lie algebroid $T_P$ can be described via $T$ as follows.
\begin{lem}Let $P$ be a commutative (dg) $A$-algebra, $T$ a Lie algebroid
over $A$ and let a map of Lie algebras and left $A$-modules
$\alpha:T\to\Der(P)$ makes commutative the following diagram
\begin{equation}
\begin{diagram}
T & \rTo^\alpha & \Der(P) \\
\dTo & & \dTo \\
\Der(A) & \rTo & \Der(A,P),
\end{diagram}
\end{equation}
where the maps to $\Der(A,P)$ are defined via composition with the
algebra map $A\to P$.
Then $\alpha$ uniquely defines a structure of $P$-Lie algebroid on
$P\otimes_AT$ and a map $(A,T)\to (P,P\otimes_AT)$ in $\Alg_{\LA}$.
\end{lem}
\begin{proof}
Straightforward.
\end{proof}

The lemma identifies $T_P$ with $P\otimes_AT$ and, in particular, defines a
Lie algebra map $S_A(T[-1])\rTo S_P(T_P[-1])$. It is, obviously,
a quasiisomorphism.

By Lemma~\ref{lem:inner} the differential in $S_A(T[-1])$ is also given
by the formula $d(x)=-\ad_{f\partial_e}$.

\subsection{Calculation}

Denote $\fg=S_A(T[-1])[1]$. One has
\begin{equation}
\begin{array}{llll}
&\fg^{-1} & = & A, \\
&\fg^0 & = & \bigoplus_{i=1}^n A\partial_i,\\
&\fg^1 & = & A\partial_e\oplus\bigoplus_{i,j} A\partial_i\wedge\partial_j,\\
&\fg^2 & = & \bigoplus_i A\partial_e\wedge\partial_i\oplus
\bigoplus_{i,j,k}\partial_i\wedge\partial_j\wedge\partial_k.
\end{array}
\end{equation}
Let $(R,\fm)$ be a local artinial $k$-algebra with the maximal ideal $\fm$.

\

Let $w=p\partial_e+S\in\fm\otimes\fg^1$ with
$S\in\bigoplus \fm\otimes A\partial_i\wedge\partial_j$.

One has $dw+\frac{1}{2}[w,w]=dS+[p\partial_e,S]+
\frac{1}{2}[S,S]$. The first two summands are divisible by
$\partial_e$ whereas the third simmand is not. Thus, $w$ satisfies Maurer-Cartan
equation iff
\begin{equation}
\begin{cases}
dS+[p\partial_e,S]& = 0\\
[S,S] & =  0
\end{cases}
\end{equation}
or (taking into account that $[f\partial_e,S]=0$ iff $[f,S]=0$)
\begin{equation}\label{eq:gen-solution}
\begin{cases}
[f-p,S]& = 0\\
[S,S] & =  0.
\end{cases}
\end{equation}

Let $T_A=\Der(A,A)=\oplus A\partial_i$.

Note that the commutative algebra $S_A(T_A[-1])$ endowed with the differential
$d=\ad_{f}$ identifies with the Koszul complex of $A$ constructed
on the sequence $(\partial_1f,\ldots,\partial_nf)$.

From now on we assume that $f$ is an isolated singularity, that is that
$\partial_if$ form a regular sequence. This implies that
$S_A(T[-1]),\ad_{f})$
is acyclic. Moreover, for any artinian local $(R,\fm)$ and any
$p\in\fm\otimes A$ the complex $(R\otimes S_A(T[-1]),\ad_{f-p})$
is also acyclic as a deformation of acyclic complex.
Therefore, $[f-p,S]= 0$ if and only if there exists a trivector
field $T$ on $A$ such that $S=[f-p,T]$.

This proves the following result.

\begin{prp}
Let $f\in A=k[x_1,\ldots,x_n]$ define an isolated hypersurface singularity.
A solution of Maurer-Cartan equation for a noncommutative unfolding
is given by a pair $(p,T)$ where $p\in\fm\otimes A$ and
$T\in\fm\otimes\wedge^3_A(T_A)$ satisfying the condition
\begin{equation}
[[f-p,T],[f-p,T]]=0.
\end{equation}
\end{prp}

\subsection{Quasiclassical data for NC unfoldings}

Recall that quasiclassical datum is defined as deformations of $B$-algebra
$A$ over $k[h]/(h^2)$ extendable to $k[h]/(h^3)$.

Deformations over $k[h]/(h^2)$ are described by the first cohomology of $\fg$.
Cocycles are given by pairs $(ph,Sh)$ with 
$p\in A,\ S\in\oplus A\partial_i\wedge\partial_j$ satisfying the condition
$[f,S]=0$, pairs $(p_1h,S_1h)$ and $(p_2h,S_2h)$ being homologous iff $S_1=S_2$
and $p_1-p_2\in(\partial_1f,\ldots,\partial_nf)$.

Maurer-Cartan solutions over $k[h]/(h^3)$ are described by pair $(p,S)$
where $p=p_1h+p_2h^2$ and $S=S_1h+S_2h^2$ satisfying (\ref{eq:gen-solution}).
This imposes two extra conditions on $(p_1,S_1)$:
\begin{itemize}
\item[1.] $[S_1,S_1]=0.$
\item[2.] There exists $S_2$ such that $[p_1,S_1]=[f,S_2]$.
\end{itemize}
Note that the second condition is equivalent to the condition
$[f,[p_1,S_1]]=0$ which is always fulfilled as $[f,S_1]=0$ and $[f,p_1]=0$.

Choose a vector subspace $W$ in $k[x_1,\ldots,x_n]$ such that 
$$ k[x_1,\ldots,x_n]=W\oplus (\partial_1,\ldots,\partial_n).$$
We have proven 
\begin{prp}\label{prp:QC}
Quasiclassical data for NC unfolding of an isolated hypersurface singularity
$f\in k[x_1,\ldots,x_n]$ are given by pairs $(p,S)$
where 
\begin{itemize}
\item[1.] $p\in W$.
\item[2.] $S$ is a Poisson bivector field satisfying $[f,S]=0$.
\end{itemize}
\end{prp}

\subsection{Quantization}

We doubt that any quasiclassical datum can be quantized in general.
This is, however, true for $n=3$ (this case includes the classical ADE
singularities) as shown the following
\begin{lem}Let $A=k[x_1,x_2,x_3]$. Any bivector field $S$ satisfying
$[f,S]=0$ is Poisson.
\end{lem}
\begin{proof} Recall that the differential in the Koszul complex
$(S_A(T_A[-1]),d)$ is given by the formula $dx=[f,x]$.
Let $S=[f,T]=dT$. One has
$$ [S,S]=[dT,dT]=d[T,dT]=0$$
as $[T,dT]$ is a four-vector.
\end{proof}

\begin{crl}\label{crl:unfolding}
Any quasiclassical datum for NC unfolding of a surface isolated 
singularity can be quantized.
\end{crl}


\begin{thebibliography}{MMMM}
\bibitem[BV]{BV} M.~Boardman, R.~Vogt, {\sl Homotopy invariant algebraic
structures on topological spaces}, Lecture notes in math., {\bf 347}, 1973.
\bibitem[DTT]{dtt}V.~Dolgushev, D.~Tamarkin, B.~Tsygan,
{\sl The homotopy Gerstenhaber algebra of Hochschild cochains of a regular 
algebra is formal,} J. Noncommutative geometry, {\bf 1} (2007), 1--25.
\bibitem[GJ]{gj} E.~Getzler, J.~Jones, {\sl Operads, homotopy aplebra and 
iterated integrals for double loop spaces,} arXiv: hep-th/9403055.
\bibitem[GK]{gk} V.~Ginzburg, M.~Kapranov, {\sl Koszul duality for operads,}
arXiv 0709.1228, Duke Math. J, {\bf 76}(1994), 203--272.
\bibitem[H0]{haha}V.~Hinich, {\sl Homological algebra of homotopy algebras},
Communications in Algebra, {\bf 25}(1997), 3291--3323.
\bibitem[H]{h-tam}V.~Hinich, {\sl Tamarkin's proof of Kontsevich formality 
theorem,} Forum Mathematicum {\bf 15}(2003) 591--614.
\bibitem[H1]{strict} V.~Hinich, {\sl Strictification of algebras over 
$\infty$-operads}, manuscript in preparation.
\bibitem[K]{ke} B.~Keller, {\sl Derived invariance of the higher structures
on the Hochschild complex,} preprint (2003) available at the author's homepage
{\tt www.math.jussieu.fr/~keller/publ/index.html}
\bibitem[L]{l} J.-L.~Loday, Cyclic homology, Grundlehren der Mathematischen 
Wissenschaften 301, Springer Berlin (1998).
\bibitem[LV]{LV} J.-L. Loday, B. Vallette, Algebraic Operads,
Grundlehren der Mathematischen Wissenschaften 346, Springer (2012).
\bibitem[Tk]{tk}D.~Tamarkin, {\sl Another proof of Kontsevich formality 
theorem,} arxiv: math/9803025.
\bibitem[TT]{tt} D.~Tamarkin, B.~Tsygan, {\sl Noncommutative differential 
calculus, homotopy BV algebras and formality conjectures}, 
arxiv: math/0002116.
\bibitem[vdL]{vdl}P.~van der Laan, {\sl Operads, Hopf algerbas and coloured 
Koszul duality,} PhD thesis, University of Utrecht, 2004.
\end{thebibliography}
\end{document}